\documentclass[12pt,a4paper]{article}
\usepackage{amsthm,amsfonts,amsmath,amssymb,bookmark,appendix,graphicx}

\theoremstyle{plain}
\newtheorem{theorem}{Theorem}
\newtheorem{lemma}[theorem]{Lemma}

\newtheorem{proposition}[theorem]{Proposition}

\theoremstyle{definition}
\newtheorem{definition}[theorem]{Definition}
\newtheorem{remark}[theorem]{Remark}
\newtheorem{example}[theorem]{Example}

\theoremstyle{remark}

\DeclareMathOperator{\Lip}{Lip}

\def\bq{\begin{eqnarray}}
\def\eq{\end{eqnarray}}
\def\bqq{\begin{eqnarray*}}
\def\eqq{\end{eqnarray*}}
\def\nn{\nonumber}
\def\minus {\backslash}

\def\R{\mathbb{R}}

\def\cF {\mathcal{F}}

\def\cR {\mathbb{R}}

\def\cN {\mathbb{N}}

\def\cL{\mathcal{L}}

\title{Kirszbraun extension on connected finite graph}
\author{Erwan Le Gruyer ~and~ Phan Thanh Viet}

\begin{document}
\maketitle

{\bf Abstract:} We prove that the tight function introduced Sheffield and Smart (2012) \cite{She} is a Kirszbraun extension. In the real-valued case we prove that Kirszbraun extension is unique. Moreover, we produce a simple algorithm which calculates efficiently the value of Kirszbraun extension  in polynomial time. 



\section{Introduction}
Let $A$ be a compact subset of $\cR^n$. The best Lipschitz constant of a Lipschitz function $g:A\to \cR^m$ is
\bq\label{eq:lipcond}
\Lip(g,A):= \sup\limits_{x\ne y\in A}\frac{\|g(x)-g(y)\|}{\|x-y\|},
\eq
where $\|.\|$ is Euclidean norm.

When $m=1$, Aronsson in 1967 \cite{Aronsson3} proved the existence of absolutely minimizing Lipschitz extension (AMLE), i.e., a extension $u$ of $g$
satisfying
\bq\label{eq:AMLEope}
\Lip(u;V)= \Lip(u,\partial V),~~ \text{for all}~~ V\subset \subset \cR^n\minus A.
\eq
Jensen in 1993 \cite{Jensen} proved the uniqueness of AMLE under certain conditions.

In this chapter we begin by studying the discrete version of the existence and uniqueness of AMLE for case $m\ge 2$. 

We define the function 
\bq
\lambda(g,A)(x) := \inf\limits_{y\in\cR^m}\sup\limits_{a\in A}\frac{\|g(a)-y\|}{\|a-x\|} \text{ if }x\in \cR^n\backslash A. 
\eq
From Kirszbraun theorem (see \cite{Federer,Kirs}) the function $\lambda(g,A)$ is well-defined and $$\lambda(g,A)(x)\le \Lip(g,A).$$ Moreover, (see \cite[Lemma 2.10.40]{Federer}) for any $x\in \cR^n\minus A$ there exists a unique $y(x)\in \cR^m$ such that 
\bq
\lambda(g,A)(x) = \sup\limits_{a\in A}\frac{\|g(a)-y(x)\|}{\|a-x\|},
\eq
and $y(x)$ belongs to the convex hull of the set
\bqq
B=\{g(z): z\in A \text{ and } \frac{\|g(z)-y(x)\|}{\|z-x\|}=\lambda(g,A)(x)\}.
\eqq
Thus we can define  
\bq\label{def:k} K(g,A)(x) := \left\{ \begin{gathered}
  g(x) \text{ if }x\in A; \hfill \\
  \arg \min\limits_{y\in\cR^m}\sup\limits_{a\in A}\frac{\|g(a)-y\|}{\|a-x\|} \text{ if }x\in \cR^n\backslash A. \\ 
\end{gathered}  \right.\eq
We say that $K(g,A)(x)$ is {\it the Kirszbraun value of $g$ restricted on $A$ at point $x.$} The function $K(g,A)(x)$ is the best extension at point $x$ such that the Lipschitz constant is minimal. We produce a method to compute $\lambda(g,A)(x)$ and $K(g,A)(x)$ in section \ref{numerical}.\\

Let $G=(V,E,\Omega)$ be a connected finite graph with vertices set $V\subset\cR^n$, edges set $E$ and a non-empty set $\Omega\subset V$.
\begin{figure}[ht]
\begin{center}
\includegraphics[scale=0.4]{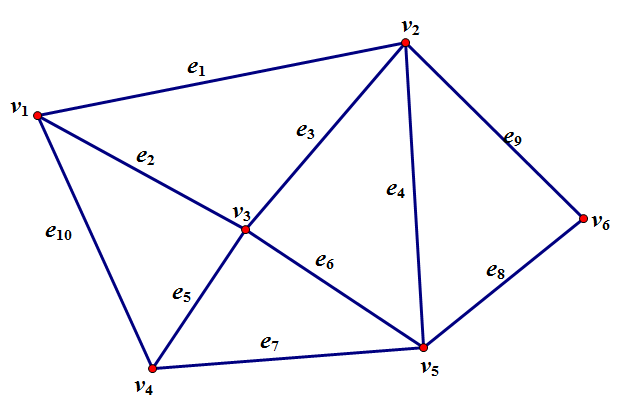}
\end{center}
\caption{A simple picture of graph $G$}
\label{fig:3.1}
\end{figure}

For $x\in V$, we define 
\bq \label{eq:sx}
S(x):=\{y\in V:(x,y)\in E\}
\eq
to be the neighborhood of $x$ on $G$.
\begin{example}
In Figure \ref{fig:3.1} we have  $V=\{v_1,...,v_6\},$ $E=\{e_1,...,e_{10}\},$ $S(v_3)=\{v_1,v_2,v_4,v_5\}.$
\end{example}

Let $f:\Omega\to \R^m.$ We consider the following functional equation with Dirichlet's condition:
\bq\label{eq:kirsz} u(x)=\left\{ \begin{gathered}
   K(u,S(x))(x) ~~\forall x\in V\minus\Omega;\hfill \\
   f(x) ~~\forall x\in \Omega. \hfill \\ 
\end{gathered}  \right.\eq

We say that a function $u$ satisfying (\ref{eq:kirsz}) is a {\it Kirszbraun extension} of $f$ on graph $G.$ This extension is the optimal Lipschitz extension of $f$ on graph $G$ since for any $x\in V\minus \Omega$, there is no way to decrease $\Lip(u,S(x))$ by changing the value of $u$ at $x.$ 

In real valued case $m=1$, the function $K(u,S(x))(x)$ was considered by Oberman \cite{Obe} and he used this function to obtain a convergent difference scheme for the AMLE. Le Gruyer \cite{ELG3} showed the explicit formula for $K(u,S(x))(x)$ as follows
\bq\label{eq:nK(u,S(x))}
K(u,S(x))(x)=\inf\limits_{z\in S(x)}\sup\limits_{q\in S(x)}M(u,z,q)(x),
\eq
where 
\bqq
M(u,z,q)(x):=\frac{\|x-z\|u(q)+\|x-q\|u(z)}{\|x-z\|+\|x-q\|}.
\eqq

Le Gruyer studied the solution of (\ref{eq:kirsz}) on a network where $K(u,S(x))(x)$ satisfying (\ref{eq:nK(u,S(x))}). This solution plays an important role in approximation arguments for AMLE in Le Gruyer \cite{ELG3}. 

The Kirszbraun extension $u$ is a generalization of the solution in the previous works of Le Gruyer for vector valued cases ($m\ge 2$). 
We prove that the tight function introduced by Sheffield and Smart (2012) \cite{She} is a Kirszbraun extension. Therefore, we have the existence of a Kirszbraun extension, but in general Kirszbraun extension maybe not unique.

In the scalar case $m=1$, Le Gruyer \cite{ELG3} defined a network on a metric space $(X,d)$ as follows
\begin{definition}\label{def:netwo} A network on a metric space $(X,d)$ is a couple $(N,U)$ where $N\subset X$ denotes a finite non-empty subset of $\cR^n$ and $U$ a mapping $x\in N\to U(x)\subset N$ which satisfies

{\it (P1)} For any $x\in N$, $x\in U(x)$.

{\it (P2)} For any $x,y\in N$, $x\in U(y)$ iff $y\in U(x)$.

{\it (P3)} For any $x,y\in N,$ there exists $x_1,...,x_{n-1}\in G$ such that $x_1=x,$ $x_n=y$ and $x_i\in U(x_{i+1})$ for $i=1,...,{n-1}.$

{\it (P4)} For any $x\in N$, any $y\in N\minus U(x)$ there exists $z\in U(x)$ such that $d(z,y)\le d(x,y).$
 
\end{definition}

In the above definition, $U(x)$ is called the neighborhood of $x$ on network $(N,U)$. Let $g: A\subset X\to \cR$. In \cite{ELG3} Le Gruyer defined the Kirszbraun extension of $g$ with respect to the network (see \cite[page 30]{ELG3}) and he proved the existence and uniqueness of the Kirszbraun extension of $g$ on the network. In particular, when $X=\cR^n$ equipped with the Euclidean norm, Le Gruyer obtained the approximation for AMLE by a sequence Kirszbraun extensions $(u_n)$ of networks $(N_n ,U_n )$ ($n\in \cN$) having some good properties.

Similarly to Le Gruyer's result about the uniqueness of the Kirszbraun extension on a network, in this chapter we prove the uniqueness of the  Kirszbraun extension $u$ of $f$ on graph $G$ when $m=1$. The graph is more general than the network in some sense since there are many graphs that do not satisfy (P4). Moreover, in the scalar case $m=1$, we produce a simple algorithm which calculates efficiently the value of Kirszbraun extension $u$ in polynomial time. This algorithm is similar to the algorithm produced by Lazarus el al. (1999) \cite{Lazarus} when they calculate the Richman cost function. Assuming Jensen's hypotheses \cite{Jensen}, since this algorithm computes exactly solution of (\ref{eq:kirsz}) and by using the argument of Le Gruyer \cite{ELG3} (the approximation for AMLE by a sequence Kirszbraun extensions $(u_n)$ of networks $(N_n ,U_n )$ ($n\in \cN$)), we obtain a new method to approximate the viscosity solution of Equation $\Delta_\infty u=0$ under Dirichler's condition $f$.

In the above algorithm, the explicit formula of $K(u,S(x))$ in (\ref{eq:nK(u,S(x))}) and the order structure of real number set play important role. The generalization of the algorithm 
to vector valued case ($m\ge 2$) is difficult since we do not know the explicit formula of $K(u,S(x))$ when $m\ge 2$ and $\R^2$ does not have any useful order structure.
Extending the results of the approximation of AMLE to vector valued cases ($m\ge 2$) presents many difficulties which have limited the number of results in this direction, see \cite {ELG2} and the references therein.

\section{The existence of Kirszbraun extension}
In this section, we prove the existence of Kirszbraun extension satisfying Equation (\ref{eq:kirsz}). 

Let $G=(V,E,\Omega)$ be a connected finite graph with vertices set $V\subset\cR^n$, edges set $E$ and a non-empty set $\Omega\subset V$ and let $f:\Omega\to \R^m$. 

We denote $E(f)$ to be the set of all extensions of $f$ on $G$.

Let $v\in E(f)$. The {\it local Lipschitz constant} of $v$ at vertex $x\in V\minus\Omega$ is given by
\bqq
Lv(x):=\sup\limits_{y\in S(x)}\frac{\|v(y)-v(x)\|}{\|y-x\|},
\eqq
where $S(x)$ is neighborhood of $x$ on $G$.
\begin{definition} \footnote{By convention, if $C=\emptyset$ then $\max\limits_{C}=0$.} If $u,v\in E(f)$ satisfy
\bqq
\max\{Lu(x):Lu(x)>Lv(x),x\in V\minus \Omega\}>\max\{Lv(x):Lv(x)>Lu(x),x\in V\minus \Omega\},
\eqq 
then we say that $v$ is {\it tighter} than $u$ on $G.$ We say that $u$ is a {\it tight extension} of $f$ on $G$ if there is no $v$ tighter than $u$. 
\end{definition} 
\begin{theorem}\cite[Theorem 1.2]{She}\label{the1} There exists a unique extension $u$ that is tight of $f$ on $G$. Moreover, $u$ is tighter than every other extension $v$ of $f$. 
\end{theorem}

\begin{proposition}\label{pro:1afd}
Let $u\in E(f)$. Let $x\in V\minus \Omega$, we define 
\[v(y) = \left\{ \begin{gathered}
  u(y),~~\text{if }y\in V\minus\{x\} ,\hfill \\
  K(u,S(x))(x),~~\text{if } y=x.\hfill \\ 
\end{gathered}  \right.\]
 If $K(u,S(x))(x)\ne u(x)$ then $v$ is tighter than $u.$
\end{proposition}
\begin{proof}\text{}\\
{\bf *Step 1:} In this step we prove that for any $y\in V\minus \Omega$, we obtain
\bq\label{eq:bt12}
Lv(y)\le \max\{Lv(x),Lu(y)\}.
\eq
Indeed,

*If $y\notin S(x)\cup\{x\}.$ Since $v(y)=u(y)$ and $v(z)=u(z)$ for all $z\in S(y)$, we obtain $$Lv(y)=Lu(y).$$

*If $y=x$. Since $v(x)\ne u(x)$ and $v(x)$ is the Kirszbraun value of $u$ restricted on $S(x)$ at point $x$, we have $$Lv(y)< Lu(y).$$

*If $y\in S(x)$ we have

\bqq
Lv(y)&=&\max\limits_{z\in S(y)}\frac{\|v(z)-v(y)\|}{\|z-y\|}\\
&=& \max\limits_{z\in S(y)\minus\{x\}}\left\{{\frac{\|v(x)-v(y)\|}{\|x-y\|}},\frac{\|u(z)-u(y)\|}{\|z-y\|}\right\}\\
&\le& \max\{Lv(x),Lu(y)\}.
\eqq

Therefore, for any $y\in V\minus \Omega$ we have 
\bqq
Lv(y)\le \max\{Lv(x),Lu(y)\}.
\eqq
{\bf *Step 2:} In this step we prove that $v$ is tighter than $u.$ It means that we need to show that 
\bqq
\max\{Lv(y):Lv(y)>Lu(y),y\in V\minus \Omega\}<\max\{Lu(y):Lu(y)>Lv(y),y\in V\minus \Omega\} 
\eqq
Indeed, if $Lv(y)>Lu(y)$ then from (\ref{eq:bt12}) we have $Lv(y)\le Lv(x)$. Thus
\bq\label{eq:123}
\max\{Lv(y):Lv(y)>Lu(y),y\in V\minus \Omega\}\le Lv(x)
\eq
Since $v(x)\ne u(x)$ and $v(x)$ is the Kirszbraun value of $u$ restricted on $S(x)$ at point $x$, we have 
\bq\label{eq145}
Lv(x)< Lu(x).
\eq
From (\ref{eq:123}) and (\ref{eq145}) we obtain
\bqq
\max\{Lv(y):Lv(y)>Lu(y),y\in V\minus \Omega\}
&\le& Lv(x)\\&<&Lu(x)\\
&\le&\max\{Lu(y):Lu(y)>Lv(y),y\in V\minus \Omega\} .
\eqq
\end{proof}

We obtain the existence of a Kirszbraun extension satisfying Equation (\ref{eq:kirsz}) as a consequence of the following theorem. 

\begin{theorem}
If $u$ is a tight extension of $f$, then $u$ is a Kirszbraun extension of $f$.
\end{theorem}
\begin{proof}
Let $u$ be a tight extension of $f$. 
Suppose, by contradiction, that there are some $x\in V\minus \Omega$ such that 
\bq\label{equ:1u}
K(u,S(x))(x)\ne u(x).
\eq
we define 
\[v(y) = \left\{ \begin{gathered}
  u(y),~~\text{if }y\in V\minus\{x\}, \hfill \\
  K(u,S(x)),~~\text{if } y=x.\hfill \\ 
\end{gathered}  \right.\]
By applying Proposition \ref{pro:1afd} we have $v$ tighter than $u$. This is impossible since $u$ is tight of $f$.
\end{proof}


\section{An algorithm to compute Kirszbraun extension when $m=1$}
In this section, let $G=(V,E,\Omega)$ be a connected finite graph, with vertices set $V\subset\cR^n$, edges set $E$ and a non-empty set $\Omega\subset V$. Let $f:\Omega\to \R.$

We recall some properties of Kirszbraun function introduced in (\ref{def:k}) which are useful in the proof of Theorem \ref{theo:chaho}.  
\begin{theorem} \label{lem:obe}Let $S=\{x_1,...,x_n\}\subset\cR^n$ and  $u:S\to \cR$. For each $x\in \cR^n\minus S$, we use the notation $d_i=\|x_i-x\|$, $i=1,...,n$. 

$(a)$ (see \cite[Theorem 5]{Obe}) We have
\bqq
K(u,S)(x)=\frac{d_iu(x_j)+d_ju(x_i)}{d_i+d_j},
\eqq
where $i,j$ are the indexes which satisfy
\bqq
\frac{|u(x_i)-u(x_j)|}{d_i+d_j}=\max\limits_{k,l=1}^n\left\{\frac{|u(x_k)-u(x_l)|}{d_k+d_l}\right\}.
\eqq

(b) (see\cite[Lemma 2.10.40]{Federer}) Let \bq
\lambda(u,S)(x) := \inf\limits_{y\in\cR^m}\sup\limits_{a\in S}\frac{\|u(a)-y\|}{\|a-x\|} \text{ if }x\in \cR^n\backslash S. 
\eq
then the set 
\bqq
B=\left\{u(z): z\in S \text{ and } \frac{\|u(z)-K(u,S)(x)\|}{\|z-x\|}=\lambda(u,S)(x)\right\},
\eqq
is not empty, and $K(u,S)(x)$ belongs to the convex hull of $B$.
\end{theorem}

\begin{theorem}\label{theo:chaho}
There is a unique Kirszbraun extension $u$ of $f$ on the graph $G.$ Moreover, the Kirszbraun extension $u$ of $f$ can be calculated in polynomial time.
\end{theorem}

Before proving Theorem \ref{theo:chaho}, we need the following definition
\begin{definition}
Let $G'=(V',E',\Omega)$ be a subgraph of $G$, i.e. $\Omega\subset V'\subset V$ and $E'\subset E$. Let $u'$ be a Kirszbraun extension of $f$ on $G'$ , a {\it connecting path on $G'$ with respect to $u'$} is a sequence
\bqq
v_0,e_1,v_1,...,e_n,v_n ~~(n\ge 1)   
\eqq
of distinct vertices and edges in $G$ such that

* each $e_i$ is an edge joining $v_{i-1}$ and $v_i$,

* $v_0$ and $v_n$ are in $V'$,

* for $1\le i< n$, $v_i$ is in $V\minus V'$, and

* for $1\le i\le n$, $e_i$ is in $E\minus E'$ 

We define
\bqq
c:=\frac{|u'(v_n)-u'(v_0)|}{\sum\limits_{i=1}^n\|v_i-v_{i-1}\|}.
\eqq
We say that $c$ is the {\it slope} of the connecting path $v_0,e_1,v_1,...,e_n,v_n.$
\end{definition}

\begin{proof}[\bf Proof of Theorem \ref{theo:chaho}]

We construct an increasing sequence of subgraph $G_n=(V_n,E_n,\Omega)$ of $G$ and $u_n$ which is a Kirszbraun extension of $f$ on $G_n$. We finish the algorithm with a Kirszbraun extension $u$ on $G$.\\

{\bf Step 1: Construct an increasing sequence of subgraph}\\

We begin with the trivial subgraph $G_1=(V_1,E_1,\Omega)$ where $V_1=\Omega$, $E_1=\emptyset$ and let $u_1=f$ on $\Omega.$ It is clear that $u_1$ is a Kirszbraun extension of $f$ on $G_1.$ The algorithm then proceeds in stages.

Suppose that after $n$ stages we have  an increasing sequence of subgraph $G_l=(V_l,E_l,\Omega)$ of $G$ and $u_l$ is a Kirszbraun extension of $f$ on $G_l$ for $l=1,...,n$. 

If there are no connecting paths on $G_n$ with respect to $u_n$, we go to step 2.

If there are some connecting paths on $G_n$ with respect to $u_n$. We construct $G_{n+1}$ subgraph of $G$ and $u_{n+1}$ Kirszbraun extension of $f$ on $G_{n+1}$ as follows: 

Find a connecting path $v_0,e_1,v_1,...,e_k,v_k ~~(k\ge 1) $ on $G_n$ with respect to $u_n$ with largest possible slope $c_n$.

Without loss of generality, we label the vertices of the path so that $u_n(v_k)\ge u_n(v_0).$ We define
\bq\label{eq:slope}
  u_{n+1}(x)&:=&\left\{ \begin{gathered}
  u_n(x),~~\forall x\in G_n\hfill \\ 
  u_n(v_0)+c_n\sum\limits_{j=1}^i\|v_j-v_{j-1}\|~~,\text{ if } x=v_i\hfill \text{ for } i=1,...,k-1. \\
\end{gathered}  \right.\\
   V_{n+1}&:=&V_n\cup\{v_1,...,v_{k-1}\}\nn  \\
  E_{n+1}&:=&E_n\cup\{e_1,..,e_k\}\nn   
\eq
We will show that $u_{n+1}$ is a Kirszbraun extension of $f$ on graph $G_{n+1}=(V_{n+1},E_{n+1},\Omega).$

For $x\in V_{n+1}$, let \bqq
S_{i}(x):=\{y\in V_{i}:(x,y)\in E_{i}\} ~~\text{ for } i\in\{1,...,n+1\}.
\eqq
be the neighborhood of $x$ with respect to $G_i.$

{\bf Case 1:} $x\in V_n\minus\{v_0,v_k\}.$ 

We have $S_{n+1}(x)=S_n(x),$  $u_{n+1}(z)=u_n(z)$ for all $z\in S_{n+1}(x)\cup \{x\}$ and $u_{n}(x)=K(u_{n},S_{n}(x))(x)$ since $u_n$ is Kirszbraun of $G_n$ . Thus 
\bqq
u_{n+1}(x)=K(u_{n+1},S_{n+1}(x))(x),~~ \text{ for } x\in V_n\minus\{v_0,v_k\} .
\eqq

{\bf Case 2:} $x\in\{v_1,...,v_{k-1}\}$.

Noting that $S_{n+1}(v_{i})=\{v_{i-1},v_{i+1}\}$ for all $i=1,..., k-1$. Moreover, from (\ref{eq:slope}), we have
\bqq
\frac{u_{n+1}(v_i)-u_{n+1}(v_{i-1})}{\|v_i-v_{i-1}\|}=c_n~~ ,\forall i: 1\le i\le k.
\eqq
Hence
\bqq
u_{n+1}(x)=K(u_{n+1},S_{n+1}(x))(x)~~ \forall x\in\{v_1,...,v_{n-1}\}.
\eqq

{\bf Case 3:} $x\in\{v_0,v_{k}\}.$ 

We need to prove that 
\bq\label{eq:haro}
u_{n+1}(v_0)=K(u_{n+1},S_{n+1}(v_0))(v_0).
\eq
(Proving $u_{n+1}(v_k)=K(u_{n+1},S_{n+1}(v_k))(v_k)$ is similar.)

To see (\ref{eq:haro}), we must show that

\bq\label{eq:haro1}
\sup\limits_{x\in S_{n+1}(v_0)}\frac{|u_{n+1}(x)-u_{n+1}(v_0)|}{\|x-v_0\|}\le\inf\limits_{y\in \cR^m}\sup\limits_{x\in S_{n+1}(v_0)}\frac{|u_{n+1}(x)-y|}{\|x-v_0\|}
\eq

Noting that $u_{n+1}(x)=u_n(x)$ for all $x\in S_n(v_0)\cup\{v_0\},$ $S_{n+1}(v_0)=S_n(v_0)\cup\{v_1\}$ and $c_n=\frac{|u_{n+1}(v_1)-u_{n+1}(v_0)|}{\|v_1-v_0\|}$. Moreover, since $u_n$ is a Kirszbraun extension of $f$ on $G_n$, we have 
\bqq
\sup\limits_{x\in S_{n}(v_0)}\frac{|u_{n}(x)-u_{n}(v_0)|}{\|x-v_0\|}\le\inf\limits_{y\in \cR^m}\sup\limits_{x\in S_{n}(v_0)}\frac{|u_{n}(x)-y|}{\|x-v_0\|}.
\eqq
Thus
\bqq
\sup\limits_{x\in S_{n+1}(v_0)}\frac{|u_{n+1}(x)-u_{n+1}(v_0)|}{\|x-v_0\|}
&=& \sup\limits_{x\in S_{n}(v_0)\cup\{v_1\}}\frac{|u_{n+1}(x)-u_{n+1}(v_0)|}{\|x-v_0\|}\\
&=& \max\limits_{x\in S_{n}(v_0)} \left\{\frac{|u_{n}(x)-u_{n}(v_0)|}{\|x-v_0\|},\frac{|u_{n+1}(v_1)-u_{n+1}(v_0)|}{\|v_1-v_0\|}\right\}\\
&=& \max\left\{\max\limits_{x\in S_{n}(v_0)}\frac{|u_{n}(x)-u_{n}(v_0)|}{\|x-v_0\|},c_n\right\},\\
\eqq
and
\bqq
\max\limits_{x\in S_{n}(v_0)}\frac{|u_{n}(x)-u_{n}(v_0)|}{\|x-v_0\|}
&\le& \inf\limits_{y\in \cR^m}\sup\limits_{x\in S_{n}(v_0)}\frac{|u_{n}(x)-y|}{\|x-v_0\|}\\
&\le& \inf\limits_{y\in \cR^m}\sup\limits_{x\in S_{n+1}(v_0)}\frac{|u_{n+1}(x)-y|}{\|x-v_0\|}.
\eqq

Therefore, to obtain Equation (\ref{eq:haro1}), we need to prove that
\bq \label{ine:haro}
c_n\le \frac{|u_n(x)-u_n(v_0)|}{\|x-v_0\|},
\eq
for some $x\in S_n(v_0).$

Let $\cF$ be the set of slope of connecting paths occurring in the algorithm. Remark that each edges and each vertices entered in our algorithm relate with a slope in $\cF.$ So that, for any $y\in V_n$, there exist some $x\in S_n(x)$ and  $c\in \cF$ such that  
\bq
c= \frac{|u_n(x)-u_n(y)|}{\|x-y\|}.
\eq
From above remark, to see (\ref{ine:haro}), we need to show that the sequence of slope of connecting paths occurring in the algorithm is non-increasing. We show this in our present notation. Suppose that 
\bqq
w_0,f_1,w_1,...,f_m,w_m ~~(m\ge 1)  
\eqq
is a connecting path on $G_{n+1}$ with respect to $u_{n+1}$ with slope $c_{n+1}.$ We need to prove that $c_n\ge c_{n+1}$. We assume without loss of generality that $u_{n+1}(w_0)\le u_{n+1}(w_m)$.

$ \bullet $ If $w_0$ and $w_m$ are both in $V_n$ then the connecting path on $G_{n+1}$ with respect to $u_{n+1}$ is actually the connecting path on $G_{n}$ with respect to $u_n$ . Therefore, since $c_n$ is the largest slope of connecting paths on $G_n$ with respect to $u_n$, we have  $c_{n}\ge c_{n+1}$.

$ \bullet $ If $w_0=v_i$ and $w_m=v_j$ for some $0\le i<j\le k$. We consider the path through the vertices
\bqq
v_0,...,v_{i-1},w_0,...,w_m,v_{j+1},...,v_k.
\eqq
The slope of above path is
\bqq
c=\frac{|u_n(v_k)-u_n(v_0)|}{\sum\limits_{l=1}^{i} {\|v_l-v_{l-1}\|}+\sum\limits_{l=1}^{m} {\|w_l-w_{l-1}\|}+\sum\limits_{l=j+1}^{k} {\|v_l-v_{l-1}\|}}.
\eqq
Since $c_n$ is the largest slope of connecting paths on $G_n$ with respect to $u_n$, we have $c_n\ge c$. Moreover,
\bqq
c_n=\frac{|u_n(v_k)-u_n(v_0)|}{\sum\limits_{l=1}^{k} {\|v_l-v_{l-1}\|}},
\eqq
thus we obtain
\bqq
\sum\limits_{l=1}^{m} {\|w_l-w_{l-1}\|}\ge \sum\limits_{l=i+1}^{j} {\|v_l-v_{l-1}\|}.
\eqq
Hence
\bqq
c_{n+1}=\frac{|u_{n+1}(w_m)-u_{n+1}(w_0)|}{\sum\limits_{k=1}^{m} {\|w_k-w_{k-1}\|}}=\frac{|u_{n+1}(v_j)-u_{n+1}(v_i)|}{\sum\limits_{k=1}^{m} {\|w_k-w_{k-1}\|}}\le \frac{|u_{n+1}(v_j)-u_{n+1}(v_i)|}{\sum\limits_{l=i+1}^{j} {\|v_l-v_{l-1}\|}}=c_n.
\eqq

{\bf Step 2: Completing the algorithm}\\

If there are no connecting paths on $G_n=(V_n,E_n,\Omega)$ with respect to $u_n$. Then each unlabeled vertex $v$ is connected via edges not in $E_n$ to exactly one vertex $w$ of $V_n$. We extend $u_n$ to the point $w$ by putting $u_n(w):=u_n(v).$ This completes the algorithm, and we obtains a Kirszbraun extension of $f$.

Since each stage adds at least one edge, and each stage can be accomplished by one shortest-path search for each pair of labeled vertices, this algorithm is calculated in polynomial time. \\

{\bf Uniqueness}\text{}\\

Let $u$ be the Kirszbraun extension of $f$ defined by the algorithm above and $h$ be another Kirszbraun extension of $f$. Let $v$ be the first vertex added by algorithm such that $u(v)\ne h(v)$ . 

$ \bullet $ If $v$ is added to a subgraph $G'=(V',E',\Omega)$ as part of a connecting path through the vertices
\bqq
v_0,...,v_k,...,v_n
\eqq
with slope  $c$ and $v=v_k$. 

We can assume without loss of generality that $u(v_0)\le u(v_n).$ Let 
\bqq
\cL=\{v_i: 0\le i\le n, h(v_i)\ge u(v_i),h(v_i)-h(v_{i-1})>u(v_i)-u(v_{i-1}) \}
\eqq
We prove that $\cL\ne \emptyset$. Indeed, by contradiction, suppose that $\cL= \emptyset$. Since $u(v_0)=h(v_0)$ and $\cL= \emptyset$ we must have 
\bqq
h(v_1)\le u(v_1).
\eqq
If $h(v_2)>u(v_2)$ then 
\bqq
h(v_2)-h(v_1)>u(v_2)-u(v_1).
\eqq
Hence $v_2\in \cL$. This contradicts with $\cL=\emptyset$. Thus we must have
\bqq
h(v_2)\le u(v_2).
\eqq
By induction, we have 
\bq\label{eq:cmdn}
h(v_i)\le u(v_i)~~ \forall i:0\le i\le k. 
\eq
Since $v=v_k,$ $h(v)\ne u(v)$ and (\ref{eq:cmdn}), we have $h(v_k)< u(v_k)$. Thus if $h(v_{k+1})\ge u(v_{k+1})$ then
\bqq
h(v_{k+1})-h(v_k)>u(v_{k+1})-u(v_k).
\eqq
Hence $v_{k+1}\in \cL$. This contradicts with $\cL=\emptyset$. Thus we must have
\bqq
h(v_{k+1})< u(v_{k+1}).
\eqq
By induction, we have 
\bqq
h(v_i)< u(v_i), ~~\forall k\le i\le n. 
\eqq
But we know that $h(v_n)=u(v_n)$, thus we have a contradiction.
Therefore $\cL\ne\emptyset.$

Let $v_l\in \cL.$ We have
\bq\left\{ \begin{gathered}
   h(v_l)\ge u(v_l);\hfill \\
   h(v_l)-h(v_{l-1})>u(v_l)-u(v_{l-1}).\hfill \\ 
\end{gathered}  \right.\eq
Hence 
\bq\label{eq:contra}
\Delta:=\frac{h(v_l)-h(v_{l-1})}{\|v_l-v_{l-1}\|}> \frac{u(v_l)-u(v_{l-1})}{\|v_l-v_{l-1}\|}=c\ge 0.
\eq

Set 
\bqq
S(x):=\{y\in V, (x,y)\in E\}~~, \text{ for }x\in V.
\eqq
Since $K(h,S(v_l))(v_l)=h(v_l)$, by applying Theorem \ref{lem:obe}, there exists $z_1\in S(v_l)$ such that $$\frac{h(z_1)-h(v_l)}{\|z_1-v_l\|}=\max\{\frac{h(y)-h(v_l)}{\|y-v_l\|}:y\in S(v_l)\}.$$ Thus
\bqq
\frac{h(z_1)-h(v_l)}{\|z_1-v_l\|}\ge \frac{h(v_l)-h(v_{l-1})}{\|v_l-v_{l-1}\|}=\Delta.
\eqq 
We extend {\it path of greatest} $z_1,z_2,...$ such that $z_{j+1}\in S(z_j)$ and 
\bqq
\frac{h(z_{j+1})-h(z_j)}{\|z_{j+1}-z_j\|}=\max\{\frac{h(y)-h(z_j)}{\|y-z_j\|}:y\in S(z_j)\}\ge \Delta.
\eqq
This path must terminate with a 
$z_m\in V'.$ 

Since $\Delta>0$, we have
\bqq
h(z_m)> ...> h(v_l)\ge u(v_l)\ge u(v_0).
\eqq
Thus $z_m\ne v_0.$

Finally, consider the path through the vertices
\bqq
v_0,v_1,...,v_{l},z_1,...,z_m.
\eqq
Set $z_0:=v_l$. The above path is the connecting path on $G'$ with respect to $u$. 

Moreover, $c$ is the largest slope of connecting paths on $V'$ with respect to $u$, and
$$u(v_0)=h(z_0),~~ u(z_m)=h(z_m),~~ h(z_0)=h(v_l)\ge u(v_l),$$
$$~~\frac{h(z_{i+1})-h(z_i)}{\|z_{i+1}-z_i\|} \ge \Delta,~~\frac{u(v_l)-u(v_0)}{\sum\limits_{i=0}^{l-1}\|v_{i+1}-v_i\|}=c, ~~\text{ and } \Delta> c.$$
Thus we have
\bqq
c &\ge& \frac{u(z_m)-u(v_0)}{\sum\limits_{i=0}^{l-1}\|v_{i+1}-v_i\|+\sum\limits_{i=0}^{m-1}\|z_{i+1}-z_i\|}\ge\frac{h(z_m)-h(z_0)+u(v_l)-u(v_0)}{\sum\limits_{i=0}^{l-1}\|v_{i+1}-v_i\|+\sum\limits_{i=0}^{m-1}\|z_{i+1}-z_i\|}\\
&=&\sum\limits_{i=0}^{m-1}\frac{h(z_{i+1})-h(z_i)}{\|z_{i+1}-z_i\|}.\frac{\|z_{i+1}-z_i\|}{\sum\limits_{i=0}^{l-1}\|v_{i+1}-v_i\|+\sum\limits_{i=0}^{m-1}\|z_{i+1}-z_i\|}\\
&+&\frac{u(v_l)-u(v_0)}{\sum\limits_{i=0}^{l-1}\|v_{i+1}-v_i\|}.\frac{\sum\limits_{i=0}^{l-1}\|v_{i+1}-v_i\|}{\sum\limits_{i=0}^{l-1}\|v_{i+1}-v_i\|+\sum\limits_{i=0}^{m-1}\|z_{i+1}-z_i\|}\\
&\ge& \Delta\frac{\sum\limits_{i=0}^{m-1}\|z_{i+1}-z_i\|}{\sum\limits_{i=0}^{l-1}\|v_{i+1}-v_i\|+\sum\limits_{i=0}^{m-1}\|z_{i+1}-z_i\|}
+c.\frac{\sum\limits_{i=0}^{l-1}\|v_{i+1}-v_i\|}{\sum\limits_{i=0}^{l-1}\|v_{i+1}-v_i\|+\sum\limits_{i=0}^{m-1}\|z_{i+1}-z_i\|}\\
&>& c
\eqq
The last inequality is obtained by $\Delta>c$. Thus we have a contradiction.  

$ \bullet $ If $v$ is added during the final step of the algorithm. We call $G'=(V',E',\Omega)$ to be the subgraph of $G=(V,E,\Omega)$ when we finish step 1 in the algorithm. Thus there are no connecting paths on $G'$ with respect to $u$. Therefore,  $v$ is connected via edges not in $E'$ to exactly one vertex $w$ of $V'$. 

We can find the largest connected subgraph $G''=(V'',E'',\Omega)$ satisfying

$$v,w\in V'',V''\cap V'=\{w\},~~\text { and } E''\cap E'=\emptyset.$$

From the definition of $u$, we have 
\bqq
h(w)=u(w)=u(x), ~~\forall x\in V''.
\eqq
Since $u(v)\ne h(v)$ and $h(w)=u(w)=u(v)$, we have $h(w)\ne h(v)$. Therefore, we must have $\sup\limits_{z\in V''}h(z)\ne h(w)$ or $\inf\limits_{z\in V''}h(z)\ne h(w)$. 

Suppose $\sup\limits_{z\in V''}h(z)\ne h(w)$  (we prove similar for the case $\inf\limits_{z\in V''}h(z)\ne h(w)$). Let $v_0\in V''$ such that
$$h(v_0)=\sup\limits_{z\in V''}h(z)\ne h(w).$$
Set 
\bqq
S''(x):=\{y\in V'':(x,y)\in E''\} ~~ \text{, for } x\in V''\minus\{w\},
\eqq 
and
\bqq
S(x):=\{y\in V:(x,y)\in E\} ~~ \text{, for } x\in V\minus\Omega.
\eqq 
Noting that 
\bq\label{eq:s=s''}
S(x)=S''(x),~~\forall x\in V''\minus\{w\}.
\eq

Since $G''$ is a connected graph, there exists a path through the vertices
\bqq
v_0, v_1,...,v_n,w
\eqq 
such that $v_i\in S''(v_{i-1}), \forall i\in\{1,...,n\}$ and $w\in S''(v_n).$

On the other hand, from (\ref{eq:s=s''}) and since $h$ is Kirszbraun extension, we have 
\bqq
h(v_0)=\sup\limits_{z\in V''}h(z)\ge \sup\limits_{z\in S''(v_0)}h(z)=\sup\limits_{z\in S(v_0)}h(z).
\eqq
 Thus applying Theorem \ref{lem:obe} we have 
$$h(v_0)=h(s),~~\forall s\in S(v_0).$$
In particular, we have $h(v_0)=h(v_1)$. By induction, we obtain
\bqq
h(v_0)=h(v_1)=...=h(v_n)=h(w).
\eqq
This contradicts with $h(w)\ne h(v_0)$. 
\end{proof}

\begin{remark}
Assuming Jensen's hypotheses \cite{Jensen}, since this algorithm computes exactly solution of (\ref{eq:kirsz}) and by using the argument of Le Gruyer \cite{ELG3} (the approximation for AMLE by a sequence Kirszbraun extensions $(u_n)$ of networks $(N_n ,U_n )$ ($n\in \cN$)), we obtain a new method to approximate the viscosity solution of Equation $\Delta_\infty u=0$ under Dirichler's condition $f$. 
\end{remark}

\begin{definition}For any $x,y\in V$. There exists a chain $x_1,...,x_n\in V$ such that $x_1=x,x_n=y$ and $x_i\in S(x_{i+1})$ for $i=1,...,n-1.$ To any chain we associate its length $\sum\limits_{i=1}^{n-1}\|x_i-x_j\|$. We define the geodesis metric $d_g$ of Graph $G$ by letting $d_g(x,y)$ be the infimum of the length of chains connecting $x$ and $y$.
\end{definition}
By using induction respect to increasing sequence of subgraph in the algorithm, we obtain the following theorem.
\begin{theorem} Let $u$ be the Kirszbraun extension of $f$. We have
\bqq
\sup\limits_{x,y\in V}\frac {\|u(x)-u(y)\|}{d_g(x,y)}\le \sup\limits_{x,y\in \Omega}\frac{\|f(x)-f(y)\|}{d_g(x,y)}, 
\eqq
and
\bqq
\inf\limits_{z\in \Omega}{f(z)}\le u(x)\le \sup\limits_{z\in \Omega}{f(z)},~~\forall x\in V.
\eqq
\end{theorem}

\section{Method to find $K(f,S)(x)$ in general case for any $m\ge 1$}\label{numerical}
We fix $S=\{p_1,...,p_N\}\subset\cR^n$ and $f:S\to \cR^m$ to be a Lipschitz function. Let $x\in \cR^n\backslash{S}$. We denote
\bqq
\lambda(f,S)(x):=\inf\limits_{y \in \cR^m} \sup\limits_{i \in \{1,...,N\}}\frac{\|y-f(p_i)\|}{\|x-p_i\|}.
\eqq
By applying Kirszbraun's theorem (see \cite{Federer,Kirs}) we have $\lambda \le \Lip(f,S)$.

In this section, we show a method to compute $\lambda(f,S)(x)$ and $K(f,S)(x)$ given by (\ref{def:k}). 

We recall some results that will be useful in this section.
\begin{lemma}(\cite[Lemma 2.10.40]{Federer}) \label{lemma 2.1234}There exists a unique $y(x)\in \cR^m$ such that 
\bq
\lambda(f,S)(x) = \sup\limits_{a\in S}\frac{\|f(a)-y(x)\|}{\|a-x\|},
\eq
and $y(x)$ belongs to the convex hull of the set
\bqq
B=\{f(z): z\in S \text{ and } \frac{\|f(z)-y(x)\|}{\|z-x\|}=\lambda(f,S)(x)\}.
\eqq
Moreover, from the definition of $K(f,S)(x)$, we have $K(f,S)(x)=y(x)$.
\end{lemma}

To compute the value of $K(f,S)(x)$ we need some properties of Cayley-Menger determinant. We recall some definitions and basic results.

Let $x_1,...,x_k\in \cR^n$. We define the Cayley-Menger determinant of $(x_i)_{i=1,...,k}$ as
\bqq\label{eq:gamma123}
\Gamma(x_1,...,x_k):=\det \left( {\begin{array}{*{20}{c}}
  0&1&1&{...}&1 \\ 
  1&0&{{{\left\| {{x_1} - {x_2}} \right\|}^2}}&{...}&{{{\left\| {{x_1} - {x_k}} \right\|}^2}} \\ 
  1&{{{\left\| {{x_2} - {x_1}} \right\|}^2}}&0&{...}&{{{\left\| {{x_2} - {x_k}} \right\|}^2}} \\ 
   \vdots & \vdots & \vdots & \ddots & \vdots  \\ 
  1&{{{\left\| {{x_k} - {x_1}} \right\|}^2}}&{{{\left\| {{x_k} - {x_2}} \right\|}^2}}&{...}&0 
\end{array}} \right).
\eqq
\begin{definition}
A $k-$simplex is a $k-$dimensional polytope which is the convex hull of its $k + 1$ vertices. More formally, suppose the $k + 1$ points $u_0,...,u_k\in\cR^n$ are affinely independent, which means $u_1-u_0,...,u_k-u_0$ are linearly independent. Then the $k-$ simplex determined by them is the set of points
\bqq
C=\{t_0u_0+...+t_ku_k: t_i\ge0,0\le i\le k,\sum\limits^k_{i=0}t_i=1\}.
\eqq
\end{definition}
\begin{example} A 2-simplex is a triangle, a 3-simplex is a tetrahedron.
\end{example}

The $k-simplex$ and the Cayley-Menger determinant have beautiful relations by following theorem:

\begin{theorem} \cite[Lemma 9.7.3.4]{Berger} \label{theo:bergerdet}Let $(x_i)_{i=1,...,k+2}\in\cR^n$ be arbitrary points in k-dimensional Euclidean affine space $X$. Then $\Gamma(x_1,...,x_{k+2})=0.$ A necessary and sufficient condition for $(x_i)_{i=1,...,k+1}$ to be a k-simplex of $X$ is that $\Gamma(x_1,...,x_{k+1})\ne 0.$
\end{theorem}
\begin{lemma} \label{Lemma:twopoint}Let the point $u$ lie in the convex hull of the points $q_0,q_1,...,q_s$ of $\cR^m$. If $u'$ distinct from $u$, then for some $i$: 
\bqq
\|u-q_i\|\le \|u'-q_i\|.
\eqq
\end{lemma}
\begin{proof}
Choose $H$ to be the $(m-1)-$ dimension (or hyperplane) through $u$ which is perpendicular to the segment $[u,u']$. Then for at least one value for $i$, $q_i$ must lie in the halfspace of $H$ which does not contain $u'$. Thus we have \bqq
\|u-q_i\|\le \|u'-q_i\|.
\eqq
\end{proof}
\begin{proposition}\label{pro:nume}
Suppose there exist $J\subset\{1,2,...,N\}$,  $f_0$ inside convex hull of $\{f(p_j)\}_{j\in J}$ and $\lambda_0>0$ such that 
\bqq
\|f_0-f(p_j)\|=\lambda_0\|x-p_j\|,~~\forall j\in J
\eqq 
and 
\bqq
\|f_0-f(p_i)\|\le \lambda_0\|x-p_i\|,~~\forall i\in\{1,...,N\},
\eqq
then $\lambda_0=\lambda(f,S)(x)$ and $f_0=K(f,S)(x).$
\end{proposition}
\begin{proof}
We have 
\bqq 
\lambda_0=\sup\limits_{i \in \{1,...,N\}}\frac{\|f_0-f(p_i)\|}{\|x-p_i\|}\ge \inf\limits_{y \in \cR^m} \sup\limits_{i \in \{1,...,N\}}\frac{\|y-f(p_i)\|}{\|x-p_i\|}=\lambda(f,S)(x).
\eqq
On the other hand, for any $y\in\cR^m$, by applying Lemma \ref{Lemma:twopoint} there exists $i\in J$ such that 
\bqq
\|y-f(p_i)\|\ge \|f_0-f(p_i)\|=\lambda_0\|x-p_i\|.
\eqq
Hence 
\bq \label{eq:bdtqt}
\sup\limits_{i \in \{1,...,N\}}\frac{\|y-f(p_i)\|}{\|x-p_i\|}\ge \lambda_0.
\eq
Since Inequality (\ref{eq:bdtqt}) is true for any $y\in \cR^m$, we have $\lambda(f,S)(x)\ge \lambda_0$. Thus$$\lambda(f,S)(x)=\lambda_0.$$
Therefore, we have
\bqq
\lambda(f,S)(x)=\sup\limits_{i \in \{1,...,N\}}\frac{\|f_0-f(p_i)\|}{\|x-p_i\|}.
\eqq
From Lemma \ref{lemma 2.1234} we have $f_0=K(f,S)(x)$.
\end{proof}

\text{}\\
{\bf A method to compute $K(f,S)(x)$}\\
Recall that $f:S\to \cR^m$. By applying Lemma \ref{lemma 2.1234}, we have 
\bqq
\|f(a)-K(f,S)(x)\|\le\lambda(f,S)(x)\|a-x\|,~~\forall a\in S.
\eqq
Moreover,
\bqq
B=\left\{f(a): a\in S \text{ and } \frac{\|f(a)-K(f,S)(x)\|}{\|a-x\|}=\lambda(u,S)(x)\right\},
\eqq
is not empty, and $K(f,S)(x)$ belongs to the convex hull of $B$.

Therefore, there exist $\{f(p_{i_k})\}_{k=1,...,l+1}\subset f(S)$ such that 

$(I)$  $l\le m$, where $m$ is dimension of $\cR^m$;

$(II)$  $\{f(p_{i_k})\}_{k=1,...,l+1}$ is a $l-$simplex. From Theorem \ref{theo:bergerdet}, $\{f(p_{i_k})\}_{k=1,...,l+1}$ is a $l-$simplex to be equivalent to
\bq\label{eq:kfs}
\Gamma(K(f,S)(x),f(p_{i_1})...,f(p_{i_{l+1}}))\ne 0;
\eq

$(III)$  $K(f,S)(x)$ belongs convex hull of $\{f(p_{i_k})\}_{k=1,...,l+1};$

$(IV)$
\bq \label{eq1:kfslame}  
\|K(f,S)(x)-f(p_{i_k})\|=\lambda(f,S)(x) \|x-p_{i_k}\|, ~~\forall k=1,...,l+1.
\eq

$(V)$
\bqq
\|f(a)-K(f,S)(x)\|\le\lambda(f,S)(x)\|a-x\|,~~\forall a\in S.
\eqq

From the above observations, we obtain
\begin{theorem} \label{theo:kirsjfjf}There exist $\{f(p_{i_k})\}_{k=1,...,l+1}\subset f(S)$ $(1\le l\le m$, where $m$ is dimension of $\cR^m$), $f_{i_1i_2...i_{l+1}}\in \cR^m$ and $\lambda_{i_1i_2...i_{l+1}}\in \cR$ satisfying some following properties

$(a)$ $f_{i_1i_2...i_{l+1}}$ inside convex hull of $\{f(p_{i_k})\}_{k=1,...,l+1}$.

$(b)~~
\Gamma(f(p_{i_1}),f(p_{i_2}),...,f(p_{i_{l+1}}))\ne 0.
$

$ (c)~~
\|f_{i_1i_2...i_{l+1}}-f(p_{k})\|=\lambda_{i_1i_2...i_{l+1}} \|x-p_{k}\|, ~~\forall k\in \{i_1,i_2,...i_{l+1}\}.
$

$ (d)~~
\|f_{i_1i_2...i_{l+1}}-f(p_k)\|\le \lambda_{i_1j_2...,i_{l+1}}\|x-p_k\|, ~~\forall k\in\{1,...,N\}.
$

Moreover, from Proposition \ref{pro:nume} we have $f_{i_1i_2...i_{l+1}}=K(f,S)(x)$ and $\lambda_{i_1,i_2...,i_{l+1}}=\lambda(f,S)(x)$.
\end{theorem}

Therefore, to compute the value of $\lambda(u,S)(x)$ and $K(u,S)(x)$ , we need to find $\{f(p_{i_k})\}_{k=1,...,l+1}\subset f(S)$, $f_{i_1i_2...i_{l+1}}\in \cR^m$ and $\lambda_{i_1i_2...i_{l+1}}\in \cR$ satisfying the conditions $(a)$,$(b)$,$(c)$,$(d)$. We can do that step by step as follows\\

{\bf *Step 1:} For all $i,j\in \{1,...,N\},(i\ne j) .$
Let 
\bqq
f_{ij}&:=&\frac{\|x-p_j\|}{\|x-p_i\|+\|x-p_j\|}f(p_i)+\frac{\|x-p_i\|}{\|x-p_i\|+\|x-p_j\|}f(p_j);\\
\lambda_{ij}&:=&\frac{\|f(p_i)-f(p_j)\|}{\|x-p_i\|+\|x-p_j\|}.
\eqq

We have $f_{ij}$ inside convex hull of $\{f(p_i),f(p_j)\}$ and 
\bqq
\|f_{ij}-f(p_k)\|=\lambda_{ij}\|x-p_k\|,~~\text{ for }k\in\{i,j\} .
\eqq
Test the following condition
\bq
\|f_{ij}-f(p_k)\|\le \lambda_{ij}\|x-p_k\|, ~~\forall k\in\{1,...,N\}
\eq
If $(i,j)$ satisfies the above condition, then from Proposition \ref{pro:nume} we have $f_{ij}=K(f,S)(x)$ and $\lambda_{ij}=\lambda(f,S)(x)$. We finish. If there is no $(i,j)\in \{1,...,N\},(i\ne j)$ that satisfies the above condition, then we go to step 2.

{\bf *Step 2:} For all $(i,j,k)\in \{1,...,N\} \times \{1,...,N\} \times \{1,...,N\}.$
Test the following condition 
\bq \label{eqrqt1}
\Gamma(f(p_{i}),f(p_{j}),f(p_{k}))\ne 0.
\eq
Let $A$ is the set of all $(i,j,k)$ that satisfies (\ref{eqrqt1}). We consider a $(i,j,k)\in A$. Thus from Theorem \ref{theo:bergerdet} we have 

$\bullet$ $\{f(p_{i}),f(p_{j}),f(p_{k})\}$ is $2-$simplex.

$\bullet$ For any $f_{ijk}$ inside convex hull of $\{f(p_i),f(p_j),f(p_k)\}$ we have
\bqq
\Gamma(f_{ijk},f(p_{i}),f(p_{j}),f(p_{k}))= 0.
\eqq

We consider the following equations
\bqq\left\{ \begin{gathered}\label{eq:thnf}
   \Gamma(f_{ijk},f(p_{i}),f(p_{j}),f(p_{k}))= 0; \\
   \|f_{ijk}-f(p_{l})\|=\lambda_{ijk} \|x-p_{l}\|, ~~\forall l\in \{i,j,k\}; \\
\end{gathered}  \right.\eqq
We replace $\|f_{ijk}-f(p_{l})\|$ by $\lambda_{ijk} \|x-p_{l}\|$ into the  equation 
\bqq
\Gamma(f_{ijk},f(p_{i}),f(p_{j}),f(p_{k}))= 0.
\eqq
We obtain that
\bqq
0&=&\Gamma(f_{ijk},f(p_{i}),f(p_{j}),f(p_{k}))\\
&=&\det \left( {\begin{array}{*{20}{c}}
  0& 1 & 1&1&1 \\ 
  1 & 0 & {{{\left\| {{f_{ijk}} - {f_i}} \right\|}^2}}&{{{\left\| {{f_{ijk}} - {f_j}} \right\|}^2}}&{{{\left\| {{f_{ijk}} - {f_k}} \right\|}^2}} \\ 
  1 & {{{\left\| {{f_i} - {f_{ijk}}} \right\|}^2}} & 0&{{{\left\| {{f_i} - {f_j}} \right\|}^2}}&{{{\left\| {{f_i} - {f_k}} \right\|}^2}} \\ 
  1 & {{{\left\| {{f_j} - {f_{ijk}}} \right\|}^2}}& {{{\left\| {{f_j} - {f_i}} \right\|}^2}}&0&{{{\left\| {{f_j} - {f_k}} \right\|}^2}} \\ 
  1 & {{{\left\| {{f_k} - {f_{ijk}}} \right\|}^2}}& {{{\left\| {{f_k} - {f_i}} \right\|}^2}}&{{{\left\| {{f_k} - {f_j}} \right\|}^2}}&0 
\end{array}} \right)\\
&=&\det\left( {\begin{array}{*{20}{c}}
  0 & 1& 1&1&1 \\ 
  1& 0 & {{\lambda _{ijk}}^2{{\left\| {x - {p_i}} \right\|}^2}}&{{\lambda _{ijk}}^2{{\left\| {x - {p_j}} \right\|}^2}}&{{\lambda _{ijk}}^2{{\left\| {x - {p_k}} \right\|}^2}} \\ 
  1& {{\lambda _{ijk}}^2{{\left\| {x - {p_i}} \right\|}^2}} & 0&{{{\left\| {{f_i} - {f_j}} \right\|}^2}}&{{{\left\| {{f_i} - {f_k}} \right\|}^2}} \\ 
  1 & {{\lambda _{ijk}}^2{{\left\| {x - {p_j}} \right\|}^2}} & {{{\left\| {{f_j} - {f_i}} \right\|}^2}}&0&{{{\left\| {{f_j} - {f_k}} \right\|}^2}} \\ 
  1 & {{\lambda _{ijk}}^2{{\left\| {x - {p_k}} \right\|}^2}} & {{{\left\| {{f_k} - {f_i}} \right\|}^2}}&{{{\left\| {{f_k} - {f_j}} \right\|}^2}}&0 
\end{array}} \right)\\
&=&a(x)\lambda^4+b(x)\lambda^2+c(x),
\eqq
where $a(x),b(x),c(x)$ are function only depending on $x$ and initial data  $x_l,$  ${f(p_{l})}$ for $l\in\{i,j,k\}$. 

By solving the equation
\bq\label{eq1:soga}
a(x)\lambda_{ijk}^4+b(x)\lambda_{ijk}^2+c(x)=0,
\eq
we obtain that $\lambda_{ijk}$ is a positive real root of the above polynomial. It maybe that Equation (\ref{eq1:soga}) have no any positive real root. In this case, we consider another $(i',j',k')\in A$ until Equation (\ref{eq1:soga}) with respect to $(i',j',k')$ have a positive real root. We call $L$ is the set of all positive real root of equation (\ref{eq1:soga}).

Let $\lambda_{ijk}\in L$. We find $f_{ijk}$ by solving the equations  
\bq\label{eq2:12}
\|f_{ijk}-f(p_{l})\|=\lambda_{ijk} \|x-p_{l}\|, ~~\forall l\in \{i,j,k\}.
\eq
After that, we test the condition $f_{ijk}$ in convex hull of $\{f(p_l)\}_{l\in\{i,j,k\}}$, and test the following condition
\bq\label{eq2:a12}
\|f_{ijk}-f(p_l)\|\le \lambda_{ijk}\|x-p_l\|, ~~\forall l\in\{1,...,N\}.
\eq

If we has a $\lambda_{ijk}\in L$ such that $f_{ijk}$ in convex hull of $\{f(p_l)\}_{l\in\{i,j,k\}}$ satisfying Equations (\ref{eq2:12}) and Inequalities (\ref{eq2:a12}) then from Proposition \ref{pro:nume} we have $f_{ijk}=K(f,S)(x)$ and $\lambda_{ijk}=\lambda(f,S)(x)$. We finish. If there is no $(i,j,k)\in A$ that satisfies the above conditions, then we go to step 3.

{\bf *Step 3:} By the similar way as step 2 for $(i,j,k,l)$, $(i,j,k,l,h),...$ until we can find a $(i_1,...,i_k)\subset \{1,...,N\}$ such that $f_{i_1,i_2...i_k}$ and $\lambda_{i_1j_2...i_k}$ satisfying some following properties

$(a)$ $f_{i_1i_2...i_k}$ inside convex hull of $\{f(p_{i_n})\}_{n=1,...,k}$

$(b)~~
\Gamma(f(p_{i_1}),f(p_{i_2}),...,f(p_{i_{k}}))\ne 0.
$

$ (c)~~
\|f_{i_1i_2...i_k}-f(p_{l})\|=\lambda_{i_1i_2...i_k} \|x-p_{l}\|, ~~\forall l\in \{i_1,i_2,...i_k\}.
$

$ (d)~~
\|f_{i_1i_2...i_k}-f(p_l)\|\le \lambda_{i_1j_2...,i_k}\|x-p_l\|, ~~\forall l\in\{1,...,N\}
$

By applying Proposition \ref{pro:nume}, we obtain $f_{i_1i_2...i_k}=K(f,S)(x)$ and $\lambda_{i_1,i_2...,i_k}=\lambda(f,S)(x)$.

\begin{remark}
By applying theorem \ref{theo:kirsjfjf}, this method terminates when $k=l+1\le m+1,$ where $m$ is dimension of $\cR^m.$ 
\end{remark}
\begin{remark} In step 3, when we solve $f_{i_1,i_2...i_k}$ by considering the equation 
\bqq
\Gamma(f_{i_1,i_2...i_k},f(p_{i_1}),f(p_{i_2}),...,f(p_{i_{k}}))= 0,
\eqq
by replacing $\|f_{i_1i_2...i_k}-f(p_{l})\| $ by $\lambda_{i_1i_2...i_k} \|x-p_{l}\|$, for $ l\in \{i_1,i_2,...i_k\},$ this equation is equivalent to 
\bq\label{ptcq}
a(x)\lambda_{i_1i_2...i_k}^4+b(x)\lambda_{i_1i_2...i_k}^2+c(x)=0,
\eq
where $a(x),b(x),c(x)$ are function only depending on $x$ and initial data  $x_l,$  ${f(p_{l})}$ for $l\in\{i_1,...,i_k\}$. The polynomial $a(x)\lambda_{i_1i_2...i_k}^4+b(x)\lambda_{i_1i_2...i_k}^2+c(x)$, in fact, is $2-$degree polynomial with variable $\lambda=\lambda_{i_1i_2...i_k}^2$. Therefore, we can solve Equation (\ref{ptcq}) very fast to obtain exactly the value of $\lambda_{i_1i_2...i_k}$.
\end{remark}


\end{document}